\documentclass{amsart}
\usepackage{amsmath}
\usepackage{amsfonts}
\usepackage{amssymb}
\usepackage{latexsym,amsthm}

\newtheorem{thm}{Theorem}[section]

\newtheorem{cor}[thm]{Corollary}
\newtheorem{lem}[thm]{Lemma}
\newtheorem{prop}[thm]{Proposition}

\theoremstyle{definition}
\newtheorem{defn}{Definition}[section]
\theoremstyle{remark}
\newtheorem{rem}{Remark}[section]
\newtheorem{exm}{Example}[section]

\numberwithin{thm}{section}
\numberwithin{equation}{section}

\newcommand{\mc}{\mathbb{C}}

\newcommand{\mi}{\mathbb{I}}
\newcommand{\mn}{\mathbb{N}}
\newcommand{\mz}{\mathbb{Z}}
\newcommand{\mb}{\mathbb{B}}
\newcommand{\mba}{\mathbf{a}}
\newcommand{\mbk}{\mathbf{k}}

\newcommand{\cM}{\mathcal{M}}
\newcommand{\cO}{\mathcal{O}}

\begin{document}


\title{On biholomorphisms between bounded quasi-Reinhardt domains}

\author{Fusheng Deng, Feng Rong}

\address{F. Deng: Matematisk Institutt, Universitetet i Oslo, Postboks 1053 Blindern, 0316 Oslo, Norway; and School of Mathematical Sciences, University of Chinese Academy of Sciences, Beijing 100049, China}
\email{fushengd@math.uio.no, fshdeng@ucas.ac.cn}

\address{F. Rong: Department of Mathematics, Shanghai Jiao Tong University, 800 Dong Chuan Road, Shanghai, 200240, P.R. China}
\email{frong@sjtu.edu.cn}

\thanks{The authors are partially supported by the National Natural Science Foundation of China (Grant Nos. 11371025 and 11371246).}

\begin{abstract}
In this paper we define what is called a quasi-Reinhardt domain and study biholomorphisms between such domains. We show that all biholomorphisms between two bounded quasi-Reinhardt domains fixing the origin are polynomial mappings and we give a uniform upper bound for the degree of such polynomial mappings. In particular, we generalize the classical Cartan's Linearity Theorem for circular domains to quasi-Reinhardt domains.
\end{abstract}

\maketitle

\section{Introduction}

Let $T^r=(S^1)^r$ be the torus group of dimension $r\ge 1$. Let $\rho:T^r\rightarrow GL(n,\mc)$ be a holomorphic linear action of $T^r$ on $\mc^n$ such that $\cO(\mc^n)^\rho=\mc$, where $\cO(\mc^n)^\rho$ is the algebra of $\rho$-invariant holomorphic functions on $\mc^n$.

\begin{defn}\label{D:quasi}
A domain $D\subset \mc^n$ is called a \emph{quasi-Reinhardt domain} of \emph{rank} $r$ with respect to $\rho$ if $D$ is $\rho$-invariant.
\end{defn}

Note that quasi-Reinhardt domains of rank one are quasi-circular domains, which include circular domains as a special case. And quasi-Reinhardt domains of rank $n$ are Reinhardt domains.

The classical Cartan's Linearity Theorem says that any biholomorphism between two bounded circular domains fixing the origin is linear (see e.g. \cite{Cartan, N:SCV}). In \cite{K:auto}, Kaup showed that any automorphism of a bounded quasi-circular domain fixing the origin is polynomial. This result was later generalized to more general settings by Heinzner in \cite{Heinzner92}. In \cite{Y:quasi, Y:quasi1}, Yamamori gave sufficient conditions for automorphisms of a bounded quasi-circular domain fixing the origin to be linear. In \cite{R:quasi}, the second named author introduced the so-called \textit{resonance order} and \textit{quasi-resonance order} for quasi-circular domains and gave a uniform upper bound for the degree of automorphisms of a bounded quasi-circular domain fixing the origin. In particular, a generalization of Cartan's Linearity Theorem was obtained for quasi-circular domains.

The aim of this paper is to generalize the results for automorphisms of bounded quasi-circular domains obtained in \cite{R:quasi} to biholomorphisms between bounded quasi-Reinhardt domains. Our main result is the following

\begin{thm}\label{T:main}
Let $\rho$ \textup{(resp.} $\rho')$ be a holomorphic linear action of $T^r$ \textup{(resp.} $T^s)$ on $\mc^n$, with all $\rho$-invariant \textup{(resp.} $\rho'$-invariant$)$ holomorphic functions on $\mc^n$ being constant. Let $D$ and $D'$ be two bounded quasi-Reinhardt domains with respect to $\rho$ and $\rho'$ respectively, containing the origin. Let $f=(f_1,\cdots, f_n)$ be a biholomorphism between $D$ and $D'$ fixing the origin. Then $f$ is a polynomial mapping with degree less than or equal to the quasi-resonance order of $\rho$ and $\rho'$. More precisely, each $f_i$ $(1\le i\le n)$ is a polynomial with degree less than or equal to the $i$-th quasi-resonance order of $\rho$ and $\rho'$. 
\end{thm}

Throughout this paper, we only consider bounded domains containing the origin. The plan of this paper is as follows. In Section \ref{S:group}, we recall some basic facts from the representation theory of compact Lie groups. In Section \ref{S:Jacobian}, we establish the important fact that a biholomorphism between two bounded quasi-Reinhardt domains has constant Jacobian determinant. In Section \ref{S:torus}, we consider the unitary representation on $H^2(D)$ induced by $\rho$, where $H^2(D)$ is the space of square integrable holomorphic functions on a bounded quasi-Reinhardt domain. In particular, we show that any minimal closed subspace of $H^2(D)$ which contains all irreducible submodules of $H^2(D)$ with the same character is finite dimensional. Using this fact, in Section \ref{S:quasi} we define the resonance order and the quasi-resonance order for quasi-Reinhardt domains and prove Theorem \ref{T:main}. Finally in Section \ref{S:special}, we give a more detailed discussion on several spacial cases of our main result.

\section{The representations of compact Lie groups}\label{S:group}

For the convenience of the reader and for setting up notations, we briefly recall some basics on the representation theory of compact Lie groups. For more detailed information, see e.g. \cite{BD:Representation}.

Let $K$ be a compact Lie group. A \textit{unitary representation} of $K$ is a Hilbert space $V$ with a continuous linear action $\rho:K\times V\rightarrow V$ such that $\rho(g):=\rho(g,\cdot):V\rightarrow V$ are isometries, $g\in K$.

Let $\rho:K\times V\rightarrow V$ be a finite dimensional representation of $K$. The \textit{character} of the representation is the function $\chi$ on $K$ defined by $\chi(g)=\text{trace}(\rho(g))$. Two finite dimensional representations are isomorphic if and only if their characters are equal. A function on $K$ is called a character of $K$ if it is the character of some representation of $K$. A character of $K$ is called \textit{irreducible} if it corresponds to an irreducible representation. We denote by $\hat K$ the set of all irreducible characters of $K$. It is known that any irreducible unitary representation of a compact Lie group must be of finite dimension.

Let $V$ be a unitary representation of $K$. Let $\chi$ be an irreducible character of $K$. Let $V_\chi$ be the minimal closed subspace of $V$ that contains all irreducible submodules of $V$ whose character are $\chi$. Then $V_\chi\bot V_{\chi'}$ for $\chi\neq\chi'$ and $V$ can be decomposed as
\begin{equation}\label{eqn:rep decompose}
V=\overline\bigoplus_{\chi\in \hat K}V_\chi,
\end{equation}
where $\overline\bigoplus$ means Hilbert direct sum.

If $\chi$ is a character of $K$, then it can be uniquely decomposed as $\chi=m_1\chi_1+\cdots+m_k\chi$, where $m_i$ are positive integers and $\chi_i$ are distinct irreducible characters of $K$. We denote $V_\chi$ the subspace $V_{\chi_1}\oplus\cdots\oplus V_{\chi_k}$.

For $\chi\in \hat K$, the orthogonal projection $p_\chi$ from $V$ to $V_\chi$ associated to the decomposition \eqref{eqn:rep decompose} can be constructed explicitly as follows:
\begin{equation}\label{eqn:projection to wt space}
p_\chi(v)=\text{dim}\chi\int_K \chi(g^{-1})\rho(g)v dg,
\end{equation}
where $dg$ is the Haar measure on $K$, and $\text{dim}\chi$ is the dimension of the representation associated to $\chi$.

For example, if $\chi_0$ is the trivial irreducible character of $K$, namely $\chi_0\equiv 1$, then $V_{\chi_0}=V^K$ is the subspace of $V$ consisting of $K$-invariant vectors, and the projection $p_{\chi_0}$ from $V$ to $V^K$ is given by the averaging $v\mapsto \int_K \rho(g)vdg$.

\section{The Jacobian of Biholomorphisms between invariant domains}\label{S:Jacobian}

Let $K$ and $K'$ be two compact Lie groups. Let $\rho:K\rightarrow \mc^n$ and $\rho':K'\rightarrow \mc^n$ be continuous representations with $\cO(\mc^n)^\rho=\mc$ and $\cO(\mc^n)^{\rho'}=\mc$. The aim of this section is to prove the following

\begin{prop}\label{prop:constant Jac}
Let $D$ and $D'$ be two bounded domains in $\mc^n$ containing the origin, which are $\rho$-invariant and $\rho'$-invariant respectively. Then any biholomorphism from $D$ to $D'$ fixing the origin has constant Jacobian determinant.
\end{prop}

For a bounded domain $D\subset \mc^n$ containing the origin, we denote by $\mi_D$ the constant function on $D$ with value 1, and $\cM_D=\{\varphi\in H^2(D); \varphi(0)=0\}$, where $H^2(D)$ is the Hilbert space of square integrable holomorphic functions on $D$. The inner product of $\varphi, \psi\in H^2(D)$ is given by $\langle \varphi, \psi\rangle_D:=\int_{D}\varphi\bar{\psi}dV$, where $dV$ is the Lebesgue measure on $\mc^n$.

From \cite[p.708]{Heinzner92}, we have the following

\begin{lem}\label{L:constant}
Let $D$ be as in Proposition \ref{prop:constant Jac}. Then $\cO(D)^\rho=\mc$.
\end{lem}

We first prove the following

\begin{lem}\label{lem:induc rep unitary}
Let $D$ be as in Proposition \ref{prop:constant Jac}. Then $\mc \mi_D\bot \cM_D$.
\end{lem}
\begin{proof}
For an irreducible character $\chi$ of $K$, let $V_{\chi}$ be the minimal closed subspace of $H^2(D)$ which contains all irreducible $K$-submodules with character $\chi$. By the assumption and Lemma \ref{L:constant}, we have $V_{\chi_0}=\mc\mi_D$, where $\chi_0$ is the trivial character.

Since we have the following orthogonal decomposition
$$H^2(D)=\overline\bigoplus_{\chi\in\hat K}V_{\chi},$$
it suffices to prove that $\mc\mi_D\bot V_{\chi}$ for any nontrivial irreducible character $\chi$ of $K$.

Let $P\in V_{\chi}$. We write $P=P_0+P'$ with $P_0\in \mc\mi_D$ and $P'\in\cM_D$. By \eqref{eqn:projection to wt space}, we have
\begin{equation}\label{eqn:int of parts}
p_\chi(P)=\text{dim}\chi\int_K \chi(g^{-1})(P_0(g(z))+P'(g(z)))dg=P_0+P'.
\end{equation}
Since the action of $K$ on $D$ fixes the origin, by comparing the coefficients on both sides of the second equality in \eqref{eqn:int of parts}, it is easy to see that $p_\chi(P_0)= P_0$ and $p_\chi(P')=P'$. So we have $P_0, P'\in V_\chi$, which is impossible if $P_0\neq 0$, since $V_{\chi_0}\bot V_\chi$. So we have $P=P'\in V_\chi$ and $P\bot V_{\chi_0}$.
\end{proof}

We now prove Proposition \ref{prop:constant Jac}.

\begin{proof}
Assume $f:D\rightarrow D'$ is a biholomorphism with $f(0)=0$. Let $u$ and $U$ be the Jacobian determinant of $f$ and $F:=f^{-1}$ respectively. By a standard argument using the change of variables formula, we have
\begin{equation}\label{eqn:inner}
\langle u(\psi\circ f),\varphi\rangle_D=\langle\psi,U(\varphi\circ F)\rangle_{D'},
\end{equation}
where $\varphi\in H^2(D), \psi\in H^2(D')$. In particular, we get
$$\langle u,\varphi\rangle_D=\langle\mi_{D'}, U(\varphi\circ F)\rangle_{D'}$$
for all $\varphi\in H^2(D)$. Since $f(0)=0$, $U(\varphi\circ F)\in \cM_{D'}$ as long as $\varphi\in\cM_D$. So, by Lemma \ref{lem:induc rep unitary}, $\langle u,\varphi\rangle_D=0$ for all $\varphi\in \cM_D$. Applying Lemma \ref{lem:induc rep unitary} again to $D$, we see that $u\in \mc \mi_D$ and hence is a constant.
\end{proof}

For later use, we also need the following 

\begin{lem}\cite[p.712]{Heinzner92}\label{L:origin}
Let $D$ and $D'$ be as in Proposition \ref{prop:constant Jac}. If $D$ and $D'$ are biholomorphic, then there is a biholomorphism between $D$ and $D'$ fixing the origin.
\end{lem}

\section{Torus actions on $\mc^n$ and Quasi-Reinhardt domains}\label{S:torus}

Let $\rho:T^r\rightarrow GL(n,\mc)$ be a holomorphic linear action of $T^r$ on $\mc^n$ such that $\cO(\mc^n)^\rho=\mc$. Let $D$ be a bounded quasi-Reinhardt domain in $\mc^n$ with respect to $\rho$ with $0\in D$. Note that the real Jacobian of $\rho(\lambda)$ has to be 1 for all $\lambda\in T^r$. The action of $T^r$ on $\mc^n$ preserves the Lebesgue measure on $\mc^n$. So the induced action of $T^r$ on $H^2(D)$ is a unitary representation. 

Since $T^r$ is an Abelian group, all irreducible submodules of $H^2(D)$ are one-dimensional. Thus the set $\hat T^r$ of irreducible characters of $T^r$ is indexed by $\mz^r$. The character corresponding to $\mbk=(k_1,\cdots, k_r)\in \mz^r$ is given by $\lambda=(\lambda_1, \cdots, \lambda_r)\mapsto\lambda^\mbk:=\lambda_1^{k_1}\lambda_2^{k_2}\cdots\lambda_r^{k_r}$.

For $\mbk\in \mathbb Z^r$, let
$$V_{\mbk}:=\{\varphi\in H^2(D);\ \varphi(\rho(\lambda)z)=\lambda^{\mbk}\cdot \varphi(z),\ \forall z\in\mc^n,\ \forall \lambda\in T^r\}.$$
We have
\begin{equation}\label{E:orthogonal}
V_{\mbk_1}\bot V_{\mbk_2},\ \ \ \mbk_1\neq\mbk_2,
\end{equation}
and $H^2(D)=\overline \bigoplus_{\mbk\in\mathbb Z^r}V_{\mbk}$. Note that from Lemma \ref{L:constant} we have $V_0=\mc$.

There exist unique $\mba_i=(a_{i1},\cdots, a_{ir})\in \mz^r$, $1\leq i \leq n$, and a corresponding coordinate system $z=(z_1,\cdots,z_n)$, such that the representation $\rho$ can be written as
\begin{equation}\label{eqn:torus action}
\rho(\lambda)z=(\lambda^{\mba_1}z_1, \cdots, \lambda^{\mba_n}z_n).
\end{equation}
Denote by $A$ the $n\times r$ matrix formed by $\mba_i$, $1\le i\le n$. Then we have $\textup{rank}A=r$.

Denote by $\mn$ the set of non-negative integers. Let $\alpha=(\alpha_1,\cdots,\alpha_n)\in \mn^n$, $\mba^j=(a_{1j}, \cdots, a_{nj})$, $1\le j\le r$, and set $\beta_j:=\alpha\cdot\mba^j=\alpha_1a_{1j}+\cdots+\alpha_na_{nj}$. Denote by $\beta=(\beta_1,\cdots,\beta_r)\in \mz^r$. For $\varphi(z)=z^\alpha=z_1^{\alpha_1}\cdots z_n^{\alpha_n}$, $\varphi(\rho(\lambda)z)=\lambda^\beta z^\alpha$. So $z^\alpha\in V_{\mbk}$ if and only if $\beta=\mbk$, and $V_{\mbk}$ is expanded by such monomials $z^\alpha$. For example, we have $z_i\in V_{\mba_i}$, $i=1,\cdots, n$.

\begin{prop}\label{prop:dim wt space finite}
With the above notations, we have $\text{dim}V_{\mbk}<\infty$ for all $\mbk\in\mathbb Z^r$.
\end{prop}
\begin{proof}
By \eqref{eqn:torus action}, we need to show that for each $\mbk\in \mz^r$, the following system
\begin{eqnarray}\label{eqn:dim weight space}
\begin{cases}
\alpha_1\cdot a_{11}+\alpha_2\cdot a_{21}+\cdots+\alpha_n\cdot a_{n1}=k_1 \\
\alpha_1\cdot a_{12}+\alpha_2\cdot a_{22}+\cdots+\alpha_n\cdot a_{n2}=k_2\\
\cdots\\
\alpha_1\cdot a_{1r}+\alpha_2\cdot a_{2r}+\cdots+\alpha_n\cdot a_{nr}=k_r
\end{cases}
\end{eqnarray}
has only finitely many solutions $\alpha=(\alpha_1, \alpha_2, \cdots, \alpha_n)\in\mn ^n$.

Note that $\alpha=0$ is the unique solution of \eqref{eqn:dim weight space} if $\mbk=0$ since $V_0=\mc$.

We argue by contradiction. Assume $\alpha^j=(\alpha^j_1, \cdots, \alpha^j_n)\in\mn ^n$, $j\geq 1$, is an infinite sequence of distinct solutions of \eqref{eqn:dim weight space}. Then we have $|\alpha^j|=\alpha^j_1+\cdots+\alpha^j_n\rightarrow\infty$ as $j\rightarrow\infty$.

If $\alpha^j_l\rightarrow\infty$ as $j\rightarrow\infty$ for all $l=1,\cdots, n$, then we can find some $j_0\gg 1$ such that $\alpha^{j_0}_l-\alpha^1_l>0$ for $l=1,\cdots, n$. Then $\alpha^{j_0}-\alpha^1\in \mn ^n$ is a nonzero solution of \eqref{eqn:dim weight space} with $\mbk=0$, which is a contradiction.

Now assume that some components of $\alpha^j$ remain uniformly bounded as $j\rightarrow\infty$. Without loss of generality, we can assume that $\alpha^j_{n_0+1},\cdots, \alpha^j_{n}$ are bounded, and $\alpha^j_l\rightarrow\infty$ for $l=1, \cdots, n_0$. Then $\alpha^j_1\cdot a_{1s}+\alpha^j_2\cdot a_{2s}+\cdots+\alpha^j_{n_0}\cdot a_{n_0s}$, $(s=1,\cdots, r,\ j\geq 1)$ can take only finitely many values. So there exists $\mbk^0=(k^0_1, \cdots, k^0_r)\in\mz^r$ such that infinitely many $\alpha'^j:=(\alpha^j_1, \cdots, \alpha^j_{n_0})$ solve the following system
\begin{eqnarray}\label{eqn:cut}
\begin{cases}
\alpha_1\cdot a_{11}+\alpha_2\cdot a_{21}+\cdots+\alpha_{n_0}\cdot a_{n_01}=k^0_1, \\
\alpha_1\cdot a_{12}+\alpha_2\cdot a_{22}+\cdots+\alpha_{n_0}\cdot a_{n_02}=k^0_2,\\
\cdots\\
\alpha_1\cdot a_{1r}+\alpha_2\cdot a_{2r}+\cdots+\alpha_{n_0}\cdot a_{n_0r}=k^0_r.
\end{cases}
\end{eqnarray}
As argued above, this implies that the homogeneous system associated to \eqref{eqn:cut}, and hence the homogeneous system associated to \eqref{eqn:dim weight space} have nonzero solutions in $\mn^{n_0}$ and $\mn^n$ respectively. This contradiction completes the proof.
\end{proof}

\begin{rem}
From the proof of Proposition \ref{prop:dim wt space finite}, we see that $\text{dim}V_{\mbk}<\infty$ for all $\mbk\in \mz^r$ is indeed equivalent to $V_0=\mc$.
\end{rem}

\section{Biholomorphisms between quasi-Reinhardt domains}\label{S:quasi}

For any $\mbk\in \mz^r$, we know by Proposition \ref{prop:dim wt space finite} that $V_{\mbk}$ consists of finitely many polynomials. Let $d_{\mbk}$ (resp. $D_{\mbk}$) be the minimum (resp. maximum) of the degrees of elements in $V_{\mbk}$. It is clear that $d_{\mbk}=0$ if and only if $\mbk=0$, and $d_{\mbk}\rightarrow\infty$ as $|\mbk|\rightarrow\infty$.

Set
$$K^l_\rho:=\{\mbk\in\mathbb Z^r;\ d_{\mbk}=1\},\ \ \ V^l_\rho:=\bigoplus_{\mbk\in K^l_\rho}V_{\mbk}.$$ 
The \textit{resonance order} of $\rho$ is defined as
$$\mu_\rho=\max\{D_{\mbk};\ \mbk\in K^l_\rho\}.$$

A more explicit description can be given as follows. For $1\le i\le n$, define the \textit{i-th resonance set} as
$$E_i:=\{\alpha;\ \alpha\cdot \mba^j=\mba_{ij},\ 1\le j\le r\},\ \ \ i.e.\ z^\alpha\in V_{\mba_i},$$
and the \textit{i-th resonance order} as
\begin{equation}\label{E:mui}
\mu_i:=\max\{|\alpha|;\ \alpha\in E_i\}=D_{\mba_i}.
\end{equation}
Then, define the \textit{resonance set} as
$$E:=\bigcup\limits_{i=1}^n E_i,$$
and the \textit{resonance order} as
$$\mu_\rho:=\max\{|\alpha|;\ \alpha\in E\}=\max\limits_{1\le i\le n} \mu_i.$$
Note that $\mu_i$ and $\mu_\rho$ are determined uniquely by $\mba_i$'s as given in \eqref{eqn:torus action}.

Let $\rho'$ be a linear action of $T^s$ on $\mc^n$ such that $\cO(\mc^n)^{\rho'}=\mc$. Let $D'\subset\mc^n$ be a bounded quasi-Reinhardt domain with respect to $\rho'$ and $0\in D'$. Then all definitions for $\rho$ carry over to $\rho'$.

We call
$$K_{\rho\rho'}:=\{\mbk\in\mathbb Z^r;\ d_\mbk\leq \mu_{\rho'}\}$$ 
the \emph{quasi-resonance set} of $\rho$ and $\rho'$, 
$$V_{\rho\rho'}:=\bigoplus_{\mbk\in K_{\rho\rho'}}V_{\mbk}$$
the \emph{quasi-resonance space} of $\rho$ and $\rho'$, and 
$$\nu_{\rho\rho'}:=\max\{D_{\mbk};\ \mbk\in K_{\rho\rho'}\}$$ 
the \emph{quasi-resonance order} of $\rho$ and $\rho'$.

More explicitly, for $1\le i\le n$, define the \textit{i-th quasi-resonance set} of $\rho$ and $\rho'$ as 
$$K_{i\rho'}=\{\mbk\in\mz^r;\ d_{\mbk}\leq \mu'_i\},$$
the \textit{i-th quasi-resonance space} of $\rho$ and $\rho'$ as
$$V_{i\rho'}:=\bigoplus_{\mbk\in K_{i\rho'}}V_{\mbk},$$
and the \textit{i-th quasi-resonance order} of $\rho$ and $\rho'$ as
$$\nu_{i\rho'}:=\max\{D_{\mbk};\ \mbk\in K_{i\rho'}\}.$$
Then the quasi-resonance order is just
$$\nu_{\rho\rho'}=\max\limits_{1\le i\le n} \nu_{i\rho'}.$$
When $\rho=\rho'$, we will drop $\rho'$ from the subscript of all the above notations.

Let $f=(f_1,\cdots,f_n):D\rightarrow D'$ be a biholomorphism fixing the origin. Let $u$ and $U$ be the Jacobian of $f$ and $F=f^{-1}$ respectively. By Proposition \ref{prop:constant Jac}, we know that both $u$ and $U$ are constants. By \eqref{eqn:inner}, we have for $\varphi\in H^2(D)$ and $1\le i\le n$
$$\langle uf_i, \varphi\rangle_D=\langle z_i, U\varphi\circ F\rangle_{D'}.$$
Since $F(0)=0$, we have $\langle z_i, U\varphi\circ F\rangle_{D'}=0$ if the zero order of $\varphi$ at $0$ is bigger than $\mu'_i$ by \eqref{E:orthogonal} and \eqref{E:mui}. Thus $\langle f_i, \varphi\rangle_D=0$ if $\varphi\in V_{\mbk}$ with $\mbk\notin K_{i\rho'}$. Hence $f_i\in V_{i\rho'}$ and the degree of $f_i$ is less than of equal to $\nu_{i\rho'}$. This completes the proof of Theorem \ref{T:main}.

Combining Theorem \ref{T:main} and Lemma \ref{L:origin}, we get

\begin{cor} 
Let $\rho$, $\rho'$, $D$ and $D'$ be as in Theorem \ref{T:main}. If $D$ and $D'$ are biholomorphic, then there is a polynomial biholomorphism between $D$ and $D'$ such that the degree of $f$ is less than or equal to the quasi-resonance order of $\rho$ and $\rho'$.
\end{cor}

In the special case that $\mu_\rho=\mu_{\rho'}=1$, it is clear that $\nu_{\rho\rho'}=1$. Hence, we have the following generalization of the classical Cartan's Linearity Theorem.

\begin{cor}\label{cor:biho qreinh reso=1}
Let $D$ and $D'$ as in Theorem \ref{T:main}. Assume that the resonance orders $\mu_\rho$ and $\mu_{\rho'}$ of $\rho$ and $\rho'$ are equal to one. Then any biholomorphism between $D$ and $D'$ fixing the origin is linear. And $D$ and $D'$ are biholomorphically equivalent if and only if they are linearly equivalent.
\end{cor}

\begin{exm}
Let $\mb^2$ be the unit ball in $\mc^2$ given by $|z_1|^2+|z_2|^2<1$. Consider biholomorphic maps $\phi_k$ given by $(w_1,w_2)=\phi_k(z_1,z_2)=(z_1,z_2+z_1^k)$, $k\in \mz^+$. Then the images $\Omega_k=\phi_k(\mb^2)$ are given by $|w_1|^2+|w_2-w_1^k|^2<1$. One readily checks that $\Omega_k$ is a quasi-circular domain with weight $(1,k)$ for each $k$. Hence for each $k$, $\phi_k$ gives a biholomorphic map between a quasi-Reinhardt domain of rank two and a quasi-Reinhardt domain of rank one.
\end{exm}

In view of the above example and Theorem \ref{T:main}, we define the \textit{maximal rank} of a quasi-Reinhardt domain $D$ to be the maximum of the ranks of all the quasi-Reinhardt domains biholomorphic to $D$.

\section{Two special cases of the main theorem}\label{S:special}

In this section, we consider two special cases of Theorem \ref{T:main}.

\subsection{Biholomorphisms between quasi-circular domains}\label{SS:quasi-circular}

Let $\rho$ and $\rho'$ be two holomorphic linear actions of $S^1$ on $\mc^n$ given by $\rho(\lambda)(z)=(\lambda^{m_1}z_1, \cdots, \lambda^{m_n}z_n)$ and $\rho'(\lambda)(w)=(\lambda^{m'_1}w_1, \cdots, \lambda^{m'_n}w_n)$. Assume that $0<m_1\le \cdots\le m_n$ and $0<m'_1\le \cdots\le m'_n$. Note that $\cO(\mc^n)^\rho=\mc$ and $\cO(\mc^n)^{\rho'}=\mc$. Let $D$ and $D'$ be two bounded quasi-circular domains containing the origin in $\mc^n$ with respect to $\rho$ and $\rho'$ respectively.

For any $k\in\mathbb N$, it is easy to see that $d_k\geq \frac{k}{m_n}$ and $D_k\leq\frac{k}{m_1}$. Thus, $D_k\leq \frac{m_n}{m_1}d_k$ and $\mu_{\rho'}\leq \frac{m'_n}{m'_1}$. Hence, the quasi-resonance order $\nu_{\rho\rho'}$ of $\rho$ and $\rho'$ satisfies the estimate 
$$\nu_{\rho\rho'}\leq \frac{m_nm'_n}{m_1m'_1}.$$

By Theorem \ref{T:main}, we have

\begin{cor}\label{cor:biholo quasi-circular}
Let $D$ and $D'$ be two quasi-circular domains as above. If $f:D\rightarrow D'$ is a biholomorphism with $f(0)=0$, then $f$ is a polynomial map whose degree is less than or equal to $\frac{m_nm'_n}{m_1m'_1}$.
\end{cor}

\subsection{Automorphisms of quasi-Reinhardt domains}

In the special case that $\rho=\rho'$ and $D=D'$, Theorem \ref{T:main} and Corollary \ref{cor:biho qreinh reso=1} read as follows.

\begin{thm}\label{thm:deg est auto qu-reinhar}
Let $D\subset\mc^n$ be a bounded quasi-Reinhardt domain with respect to $\rho$ with $0\in D$. If $f=(f_1, \cdots, f_n)$ is an automorphism of $D$ fixing the origin, then $f$ is a polynomial map such that $\textup{deg} f_i\leq \nu_i\ (1\leq i \leq n)$ and $\textup{deg} f\leq \nu_\rho$.
\end{thm}

\begin{cor}\label{cor:linear}
Let $D\subset\mc^n$ be a bounded quasi-Reinhardt domain with respect to $\rho$ with $0\in D$. If $\mu_\rho=1$ then any automorphism of $D$ fixing the origin is linear.
\end{cor}

There is no simple explicit formula for $\nu_i$ in terms of the weights $\mba_i$, $i=1,\cdots,n$, as given in \eqref{eqn:torus action}. On the other hand, if all $a_{ij}$ are nonnegative, a rough estimate of $\nu_{i\rho}$ can be given as follows.

\begin{prop}\label{thm:deg esti nonneg weight}
Let $D\subset\mc^n$ be a bounded quasi-Reinhardt domain with respect to $\rho$ with $0\in D$. Assume that all $a_{ij}\geq 0, (1\leq i\leq n, 1\leq j\leq r)$ and $|\mba_1|\leq\cdots\leq |\mba_n|$. Let $f=(f_1,\cdots, f_n)$ be an automorphism of $D$ with $f(0)=0$. Then we have
$$\textup{deg} f_i\leq \frac{|\mba_i||\mba_n|}{|\mba_1|^2}$$
for $1\leq i\leq n$, and
$$\textup{deg} f\leq \frac{|\mba_n|^2}{|\mba_1|^2}.$$
\end{prop}
\begin{proof}
Note that, since $\cO(\mc^n)^\rho=\mc$, we have $|\mba_i|\neq 0$ for all $i=1, \cdots, n$. Since all $V_{\mbk}$ are expanded by monomials, if $z^\alpha=z^{\alpha_1}_i\cdots z^{\alpha_n}_n\in V_{\mbk}$, then by \eqref{eqn:dim weight space}, we have
$$\alpha_1\cdot |\mba_1|+\cdots +\alpha_n\cdot|\mba_n|=|\mbk|.$$
Thus, we get
$$d_{\mbk}\geq \frac{|\mbk|}{|\mba_n|},\ \  D_{\mbk}\leq \frac{|\mbk|}{|\mba_1|},$$
and hence
$$D_{\mbk}\leq \frac{|\mba_n|}{|\mba_1|}d_{\mbk}.$$
Therefore, for $\mbk\in K_i$, we have
$$D_{\mbk}\leq \frac{|\mba_n|}{|\mba_1|}D_{\mba_i}\leq \frac{|\mba_n|}{|\mba_1|}\frac{|\mba_i|}{|\mba_1|}.$$
The proposition then follows from Theorem \ref{thm:deg est auto qu-reinhar}.
\end{proof}

\begin{rem}
Even in the case discussed in Proposition \ref{thm:deg esti nonneg weight}, Theorem \ref{thm:deg est auto qu-reinhar} is stronger than Proposition \ref{thm:deg esti nonneg weight} and hence Proposition \ref{thm:deg esti nonneg weight} and Corollary \ref{cor:biholo quasi-circular} are far from optimal. For example, consider the action of $S^1$ on $\mc^2$ given by $\rho(e^{i\theta})(z_1,z_2)=(e^{im_1\theta}z_1, e^{im_2\theta}z_2)$, with $m_2>m_1>1$ and $\text{gcd}(m_1,m_2)=1$. Then the bound on $\textup{deg}f$ given by Theorem \ref{thm:deg est auto qu-reinhar} is one and hence $f$ is a linear map. But the bound given by Proposition \ref{thm:deg esti nonneg weight} is very big if $m_2\gg m_1$.
\end{rem}


\begin{thebibliography}{123}

\bibitem{BKU:auto}
R. Braun, W. Kaup, H. Upmeier; \emph{On the automorphisms of circular and Reinhardt domains in complex Banach spaces},
Manuscripta Math. {\bf 25} (1978), 97-133.

\bibitem{BD:Representation}
T. Br\"{o}cker, T. Dieck;
Representations of Compact Lie Groups,
Grad. Texts in Math. {\bf 98}, Springer-Verlag, New York, 1985.

\bibitem{Cartan}
H. Cartan; \emph{Les fonctions de deux variables complexes et le probl\`{e}me de la repr\'{e}sentation analytique},
J. Math. Pures. Appl. {\bf 10} (1931), 1-114.

\bibitem{Heinzner92}
P. Heinzner;
\emph{On the automorphisms of special domains in $\mc^n$},
Indiana Univ. Math. J. {\bf 41} (1992), 707-712.

\bibitem{K:auto}
W. Kaup; \emph{\"{U}ber das Randverhalten von holomorphen Automorphismen beschr\"{a}nkter Gebiete},
Manuscripta Math. {\bf 3} (1970), 257-270.

\bibitem{N:SCV}
R. Narasimhan; \emph{Several Complex Variables},
Chicago Lect. Math., The University of Chicago Press, 1971.

\bibitem{R:quasi}
F. Rong; \emph{On automorphisms of quasi-circular domains fixing the origin},
Bull. Sci. Math. {\bf 140} (2016), 92-98.

\bibitem{Y:quasi}
A. Yamamori; \emph{Automorphisms of normal quasi-circular domains},
Bull. Sci. Math. {\bf 138} (2014), 406-415.

\bibitem{Y:quasi1}
A. Yamamori; \emph{Sufficient conditions for the linearity of origin-preserving automorphisms of quasi-circular domains},
Preprint, arXiv:1404.0309.

\end{thebibliography}
\end{document}